\documentclass{article}
\usepackage[round]{natbib}
\usepackage{amsmath,amsfonts,amssymb,amsthm}
\allowdisplaybreaks
\usepackage{paralist}
\usepackage{booktabs}
\usepackage{graphicx}

\renewcommand{\emptyset}{\varnothing}
\renewcommand{\mid}{\,|\,}

\newcommand{\real}{\mathbb{R}}
\newcommand{\bsx}{\boldsymbol{x}}

\newcommand{\bsz}{\boldsymbol{z}}

\newcommand{\cx}{\mathcal{X}}
\newcommand{\e}{\mathbb{E}}

\newcommand{\val}{\mathrm{val}}
\newcommand{\var}{\mathrm{var}}
\newcommand{\cov}{\mathrm{cov}}
\newcommand{\cor}{\mathrm{cor}}

\newcommand{\ult}{\underline{\tau}}
\newcommand{\olt}{\overline{\tau}}

\newtheorem{theorem}{Theorem}
\newtheorem{proposition}{Proposition}

\newtheorem{lemma}{Lemma}

\newcommand{\rd}{\,\mathrm{d}}
\newcommand{\mrd}{\mathrm{d}}

\newcommand{\dnorm}{\mathcal{N}}
\newcommand{\dustd}{\mathbf{U}}
\newcommand{\tran}{\mathsf{T}}

\newcommand{\phe}{\phantom{{}={}}}

\renewcommand{\ge}{\geqslant}
\renewcommand{\le}{\leqslant}

\title{On Shapley value for measuring importance 
of dependent inputs}
\date{Orig: October 2016\\[1.15ex]This: March  2017}
\author{Art B. Owen\\Stanford University \and Cl\'ementine Prieur\\ Universit\'e Grenoble Alpes, CNRS, LJK, F-38000 Grenoble, France\\Inria project/team AIRSEA}

\begin{document}
\maketitle

\begin{abstract}
This paper makes the case for using Shapley value to quantify
the importance of random input variables to a function. 
Alternatives based on the ANOVA decomposition
can run into conceptual and computational problems when the 
input variables are dependent.
Our main goal here is to show that Shapley value removes the
conceptual problems. We do this with some simple examples
where Shapley value leads to intuitively reasonable nearly
closed form answers.
\end{abstract}

\section{Introduction}\label{sec:intro}

The importance of inputs to a function is commonly
measured via Sobol' indices.  Those are defined in terms
of the functional analysis of variance (ANOVA) decomposition, which is conventionally defined
with respect to statistically independent inputs.
In applications to computer experiments, it is common that the input
space is constrained to a non-rectangular region, or that
the input variables have some other known form of dependence,
such as a general Gaussian distribution.
When the inputs are described by an empirical distribution
on observational data it is
extremely rare that the variables are statistically independent.
Even designed experiments avoid having independent inputs
(i.e., a Cartesian product of input levels) when the dimension
is moderately large \citep{wu2011experiments}.

A common way to address dependence is to build on work by
\cite{ston:1994} and \cite{hook:2007} who define an ANOVA for
dependent inputs and then define variable importance through
that generalization of ANOVA.  This is the method taken by
\cite{chas:gamb:prie:2012} for computer experiments.

The dependent-variable ANOVA leads to importance measures 
with two conceptual problems:
\begin{compactenum}[\quad1)]
\item the needed ANOVA is only defined when the random 
$\bsx$ has a distribution with a density (or mass function)
uniformly bounded below by a positive constant times
another density/mass function that has independent
margins, and 
\item  the resulting importance of a variable can  be negative
\citep{chas:gamb:prie:2015}.
\end{compactenum}

The first condition is very problematic.  
It fails even for Gaussian $\bsx$ with 
nonzero correlation.  It fails for inputs 
constrained to a simplex.  It fails when 
the empirical distribution of say $(x_{i1}, x_{i2})$
is such that some input combinations are 
never observed or, by definition, cannot possibly be observed.

The second condition is also conceptually problematic.
A variable on which the function does not
depend at all will get importance zero and thus be more
important than one that the function truly does depend on
in a way that gave it negative importance.

The Shapley value, from economics, provides an alternative 
way to define variable importance.   As we describe below,
Shapley value provides a way to attribute the value created
by a team to its individual members.  In our context
the members are individual input variables.
\cite{sobolshapley} derived Shapley value importance for 
independent inputs where the value is variance explained.
The Shapley value of a variable turns out to be bracketed between 
two different Sobol' indices. \cite{song:nels:staum:2016}
recently advocated the use of Shapley value for the case of dependent inputs.  
They report that it is more suitable than Sobol' indices for such problems. 
They use the term ``Shapley effects'' to describe variance based Shapley values. 

The Shapley value provides an importance
measure that avoids the two problems mentioned above: It is available for
any function in $L^2$ of the appropriate domain and it never
gives negative importance.

Although Shapley value solves the conceptual problems,
computational problems remain a serious challenge
\citep{castro2009polynomial}.  The Shapley
value is defined in terms of $2^d-1$ models where $d$ is the
dimension of $\bsx$. 
\cite{song:nels:staum:2016}
presented a Monte Carlo algorithm to estimate
Shapley importance and they apply it to detailed real-world problems.
We address only the conceptual appropriateness of Shapley value
to variable importance, not computational issues.

The outline of this paper is as follows.
Section~\ref{sec:notation} introduces our notation,
defines the functional ANOVA and the Sobol' indices
and presents the dependent-variable ANOVA.
Section~\ref{sec:shapley}
 presents the Shapley value
and its use for variable importance.
From the definition there it is clear that Shapley
value for variance explained will never be negative.
Section~\ref{sec:examples} gives several examples of
simple cases and exceptional corner cases where
we can derive the Shapley value of variable importance
and verify that it is reasonable.
Section~\ref{sec:conc} has brief conclusions.
Section~\ref{sec:proofs} contains the longer proofs.

\section{Notation}\label{sec:notation}
We consider real valued functions $f$ defined on a space $\cx$.
The point $\bsx\in\cx$ has $d$ components, and we write
$\bsx =(x_1,\dots,x_d)$ where $x_j\in\cx_j$.
The individual $\cx_j$ are ordinarily interval subsets of $\real$ but
each of them may be much more general (regions in Euclidean
space, functions on $[0,1]$, or even images, sounds, and video).
What we must assume is that $\bsx$ follows a distribution
$P$ chosen by the user, and that $f(\bsx)$ is then a random
variable with $\e(f(\bsx)^2)<\infty$.

When the components of $\bsx$ are independent,
then Sobol' indices \citep{sobo:1990,sobo:1993} provide  ways to measure 
the importance of individual components of $\bsx$
as well as sets of them.
They are based on a functional ANOVA decomposition.
For details and references on the functional ANOVA,
see~\cite{sobomat}.  

\subsection{ANOVA for independent variables}

Here is a brief summary of the ANOVA to introduce our notation.
For simplicity we will take $f\in L^2[0,1]^d$ with
the argument $\bsx=(x_1,\dots,x_d)$ of $f$ uniformly
distributed on $[0,1]^d$, but the approach
extends straightforwardly to $L^2(\prod_{j=1}^d\cx_j)$
with independent not necessarily uniform $x_j\in\cx_j$.

The set $\{1,2,\dots,d\}$ is written $1{:}d$.
For $u\subseteq 1{:}d$, $|u|$ denotes cardinality
and $-u$ is the complement $\{1\le j\le d\mid j\not\in u\}$.
If $u=(j_1,j_2,\dots,j_{|u|})$ then
$\bsx_u = (x_{j_1},x_{j_2},\dots,x_{j_{|u|}})\in[0,1]^{|u|}$
and $\mrd\bsx_u = \prod_{j\in u}\mrd x_j$.
We use $u+v$ as a shortcut for $u\cup v$ when $u\cap v=\emptyset$,
especially in subscripts.

The ANOVA is defined via functions $f_u\in L^2[0,1]^d$.
These functions satisfy $f(\bsx) = \sum_{u\subseteq 1{:}d}f_u(\bsx)$.
They are defined as follows. First,
$f_\emptyset = \int f(\bsx)\rd\bsx$ and  then
\begin{align}\label{eq:deffu}
f_u(\bsx) = \int\bigl( f(\bsx)-\sum_{v\subsetneq u}f_v(\bsx) \bigr)\rd\bsx_{-u}
\end{align}
for $|u|>0$.
The integral in~\eqref{eq:deffu} is over $[0,1]^{d-|u|}$ and it yields
a function $f_u$ that depends on $\bsx$ only through $\bsx_u$.
The effects $f_u$ are orthogonal: $\int f_u(\bsx)f_v(\bsx)\rd\bsx=0$
when $u\ne v$.

The variance component for the set $u$ is
$\sigma^2_u=\int f_u(\bsx)^2\rd\bsx$ for $|u|>0$
and $\sigma^2_\emptyset=0$.
The variance of $f$ for $\bsx\sim\dustd[0,1]^d$
is $\sigma^2 = \sum_{u\subseteq1{:}d}\sigma_u^2$.

We can define the importance of a set of variables
by  how much of the variance of $f$ is explained by those variables.
The best prediction of $f(\bsx)$ given $\bsx_u$ is
$$ f_{[u]}(\bsx) \equiv \e( f(\bsx)\mid \bsx_u)  = \sum_{v\subseteq u}f_v(\bsx).$$
This prediction explains
\begin{align}\label{eq:varexpl}
\ult^2_u \equiv \sum_{v\subseteq u}\sigma^2_v,
\end{align}
of the variance in $f$. This is one of Sobol's global sensitivity
indices.   His other index is
$$\olt^2_u \equiv \sum_{v\cap u\ne \emptyset}\sigma^2_v = \sigma^2 - \ult^2_{-u}.$$
It is more conventional to use normalized versions
$\ult^2_u/\sigma^2$ and $\olt^2_u/\sigma^2$
but unnormalized ones are simpler for our purposes.
The importance of an individual variable $x_j$ is sometimes defined
through $\ult^2_{\{j\}}$ or $\olt^2_{\{j\}}$.
If $\ult^2_{\{j\}}$ is large then $x_j$ is important and if $\olt^2_{\{j\}}$ is
small then $x_j$ is unimportant.

\subsection{ANOVA for dependent variables}
Now suppose that $f$ is defined on 
$\real^d$ but the argument $\bsx$ 
does not have independent components.
Instead $\bsx$ has distribution $P$.
We could generalize~\eqref{eq:deffu} to the Stone-Hooker ANOVA
\begin{align}\label{eq:deffug}
f_u(\bsx) = \int\bigl( f(\bsx)-\sum_{v\subsetneq u}f_v(\bsx) \bigr)\rd P(\bsx_{-u})
\end{align}
but the result would not generally have orthogonal effects.
To take a basic example, suppose that $P$ is the 
$\dnorm
\left(
\left(\begin{smallmatrix} 0\\0 \end{smallmatrix}\right),
\left(\begin{smallmatrix} 1 &\rho\\ \rho &1 \end{smallmatrix}\right)
\right)
$
distribution for $0<\rho<1$ and let $f(\bsx) = \beta_1x_1+\beta_2x_2$.
Then~\eqref{eq:deffug} yields
$$
f_\emptyset(\bsx)=0,\quad 
f_{\{1\}}(\bsx) = (\beta_1+\beta_2\rho)x_1,\quad
f_{\{2\}}(\bsx) = (\beta_2+\beta_1\rho)x_2
$$
and
$f_{\{1,2\}}(\bsx)=-\beta_2\rho x_1 -\beta_1\rho x_2.
$ 
These effects  are not  orthogonal under $P$ and their
mean squares do not sum to the variance of $f(\bsx)$
for $\bsx\sim P$.

It is however possible to get a decomposition 
$f(\bsx) = \sum_{u\subseteq1{:}d}f_u(\bsx)$
with a hierarchical orthogonality property
\begin{align}\label{eq:hop}
\int f_u(\bsx)  f_v(\bsx) \rd P(\bsx) = 0,\quad \forall  v\subsetneq u.
\end{align}
\cite{chas:gamb:prie:2012}
give conditions under which a decomposition
of $f$ satisfying~\eqref{eq:hop} exists and they use it to define
variable importance.

They assume that the joint distribution $P$ is absolutely continuous
with respect to a product probability measure $\nu$.
That is $P(\mrd \bsx) = p(\bsx)\prod_{j\in1{:}d}\nu_j(\mrd x_j)$
for a density function $p$.
They require also that this density satisfies
\begin{align}\label{eq:noholes}
\exists\, 0<M\le 1,\quad\forall u\subseteq1{:}d,\quad
p(\mrd\bsx) \ge M p(\mrd\bsx_u)p(\mrd\bsx_{-u}),\quad\nu-\text{a.e.}
\end{align}
The joint density is bounded below by a product of two marginal
densities. Among other things, this criterion forbids `holes' in
the support of $P$.  There cannot be regions $R_u\in\real^u$
and $R_{-u}\in\real^{-u}$ with $P(R_u\times R_{-u})=0$
while $\min( P(R_u\times\real^{-u}),P(\real^u\times R_{-u}))>0$.

\subsection{Challenges with dependent variable ANOVA}

The no holes condition~\eqref{eq:noholes} is problematic in many
applications. 
For example, when $\bsx$ is uniformly
distributed on  the triangle
$$
\{ (x_1,x_2)\in[0,1]^2\mid x_1\le x_2\}
$$
then~\eqref{eq:noholes} is violated.
More generally, \cite{gilquin2015sobol} and \cite{kucherenko2016sobol} consider functions on non-rectangular regions defined by linear inequality constraints.
These and similar regions arise in many engineering problems where safety
or costs impose constraints on design parameters.

The simplest distribution with a hole is one with positive probability on the points
$$
\{ (0,0), (0,1), (1,0)\}
$$
and no others. 
Sobol's `pick-freeze' methods 
\citep{sobo:1990,sobo:1993} 
estimate variable importance by freezing the level of
some inputs and then picking new values for the others.  For the example here, setting
$x_1=1$ implies that $x_2$ cannot be changed at all, which is a severe problem
for a pick-freeze approach with dependent inputs.

It is not just probability zero holes that cause a problem for dependent variable ANOVA.
When $\bsx$ is normally distributed with some nonzero correlations,
then~\eqref{eq:noholes} does not hold, and then as we mentioned in
the introduction, the dependent-variable ANOVA is unavailable.
The second problem we mentioned there is that
the dependent variable ANOVA can yield negative estimates of importance.

\section{Shapley value}\label{sec:shapley}

Shapley value is a way to attribute the economic output
of a team to the indivitual members of that team.
In our case, the team will be the set of variables $x_1,x_2,\dots,x_d$.
Given any subset $u\subseteq1{:}d$ of variables, the value that
subset creates on its own is its explanatory power.  A convenient
way to measure explanatory power is via 
\begin{align}\label{eq:uval}
\val(u) = \ult^2_u \equiv \var( \e( f(\bsx)\mid \bsx_u)).
\end{align}
Here, the empty set creates no value and the entire team contributes $\sigma^2$
which we must now partition among the $x_j$.

There are four very compelling properties that an attribution
method should have.
The following list is based on the account in \cite{wint:2002}. 
Let $\val(u)\in\real$ be the value attained by the 
subset $u\subseteq\{1,\dots,d\}\equiv1{:}d$. It is 
always assumed that $\val(\emptyset)=0$, which holds in our variance explained setting.
The values $\phi_j=\phi_j(\val)$
should satisfy these properties:
\begin{compactenum}[\quad1)]
\item (Efficiency) $\sum_{j=1}^d\phi_j = \val(1{:}d)$.
\item (Symmetry) If $\val(u\cup\{i\})=\val(u\cup\{j\})$
for all $u\subseteq 1{:}d-\{i,j\}$, then $\phi_i=\phi_j$. 
\item (Dummy) If $\val(u\cup\{i\})=\val(u)$ for all $u\subseteq1{:}d$,
then $\phi_i=0$. 
\item (Additivity) If $\val$ and $\val'$
have Shapley values $\phi$ and $\phi'$ respectively 
then the game with value $\val +\val' $
has Shapley value $\phi_j+\phi'_j$ for $j\in1{:}d$.
\end{compactenum}
\medskip
\cite{shap:1953} showed that the unique valuation $\phi$
that satisfies these axioms attributes value
\begin{align*}
\phi_j 
&= \frac1d\sum_{u\subseteq -\{j\}}
{d-1\choose |u|}^{-1}
\bigl( \val(u\cup\{j\})-\val(u)\bigr)
\end{align*}
to variable $j$.
Defining the value via~\eqref{eq:uval} we get
\begin{align}
\phi_j 
&= \frac1d\sum_{u\subseteq -\{j\}}
{d-1\choose |u|}^{-1}
(\ult^2_{u+\{j\}}-\ult^2_u).\label{eq:shapfromult}
\end{align} 

From~\eqref{eq:shapfromult} we see that the Shapley value
is defined for any function for which $\var(\e(f(\bsx)\mid \bsx_u))$
is always defined.  The components $\bsx_j$ do not have to be real
valued, though $f(\bsx)$ must be.  Holes in the domain $\cx$ do not
make it impossible to define a Shapley value.
Next, because $\bsx_{u+\{j\}}$ always has at least as much explanatory power
as $\bsx_u$ has, we see that $\phi_j\ge0$. That is, no variable has
a negative Shapley value. 
As a result, the Shapley value addresses the two conceptual problems
mentioned in the introduction.

\cite{song:nels:staum:2016}
show that the same Shapley value
arises if we use $\val(u) = \e( \var(f(\bsx)\mid \bsx_{-u}))$. 
That provides an alternative way to compute Shapley
value.
The Shapley value simplifies for independent inputs.
\begin{theorem}\label{thm:shapleyshare}
Let the ANOVA decomposition of a function with $d$ independent
inputs have variance components $\sigma^2_u$ for $u\subseteq1{:}d$.
If the value of a subset $u$ of variables 
is $\val(u)=\ult^2_u$, then
the Shapley value of variable $j$ is 
$$
\phi_j = \sum_{u\subseteq 1{:}d,\; j\in u} {\sigma^2_u}/{|u|}. 
$$
\end{theorem}
\begin{proof}
\cite{sobolshapley}.
\end{proof}

It follows from Theorem~\ref{thm:shapleyshare} that $\ult^2_{\{j\}} \le \phi_j \le \olt^2_{\{j\}}$.
This is how the Sobol' indices bracket the Shapley value.

\section{Special cases}\label{sec:examples}

Here we consider some special case distributions and
toy functions where we can work out the Shapley value in
a closed or nearly closed form. The point of these examples
is to show that Shapley gives sensible answers in both
regular cases and corner cases.
Because $\sigma^2 = \var(\e(f(\bsx)\mid \bsx_u))
+ \e( \var( f(\bsx)\mid \bsx_u))$ we may use
\begin{align}\label{eq:othertau}
\ult_u^2 = \sigma^2-\e( \var(f(\bsx)\mid \bsx_u)).
\end{align}

\subsection{Linear functions}
Let $f(\bsx) = \beta_0+\sum_{j=1}^d \beta_jx_j$ where $x_j$
are independent with variances $\sigma^2_j$.  It is then easy 
to find that $\phi_j = \beta_j^2\sigma^2_j$. 
If we reparameterize $x_j$ to $cx_j$ for $c\ne 0$ then 
$\beta_j$ becomes $\beta_j/c$ and the importance of this 
variable remains unchanged as it should. 
Dependence among the  $x_j$ complicates the expression for  Shapley effects in linear settings.

Shapley value for linear functions has historically been used
to partition the $R^2$ quantity (proportion of sample variance explained) 
from a regression on $d$ variables
among those $d$ variables.  Taking the value of a subset $u$ of
variables to be $R^2_u$, the $R^2$ value when regressing
a response on predictors $x_j$ for $j\in u$, yields Shapley value
\begin{align}\label{eq:lmg}
\phi_j = \frac1d\sum_{u\subseteq -\{j\}}{d-1\choose |u|}^{-1}(R^2_{u+\{j\}}-R^2_u).
\end{align}
This is the LMG measure of variable importance, named after
the authors of~\cite{lind:mere:gold:1980}.
If we rearrange the $d$ variables into all $d!$ orders, find the improvement in $R^2$
that comes at the moment the $j$'th variable is added to the regression, then
\eqref{eq:lmg} is the average of all those improvements.
The LMG reference is difficult to obtain. \cite{geni:1993}
is another reference, having~\eqref{eq:lmg} as equation (1).
\cite{grom:2007} cites several more references on partitioning $R^2$
in regression and discusses alternative measures and criteria for choosing.
It is clear that~\eqref{eq:lmg} is expensive for large $d$.

Here we consider a population/distribution version of partitioning
variance explained among a set of variables acting linearly.
We suppose that $\bsx \sim\dnorm( \mu,\Sigma)$
where $\Sigma\in\real^{d\times d}$ is a positive semi-definite 
symmetric matrix. The function of interest is
$f(\bsx) = \beta_0+\bsx^\tran\beta$ where $\beta=(\beta_1,\dots,\beta_d)\in\real^d$.
If there is an error term as in a linear regression on noisy data,
then we can let $x_d$ be that error variable with a corresponding $\beta_d=1$.

If $\Sigma$ is not diagonal then the Stone-Hooker ANOVA 
is not available because~\eqref{eq:noholes} does not hold. 
Shapley value gives an interpretable expression for general $d$. 

\begin{theorem}\label{thm:gauslin}
If $f(\bsx)=\beta_0+\beta^\tran\bsx$ for $\bsx\sim\dnorm(\mu,\Sigma)$
where $\Sigma\in\real^{d\times d}$ has full rank, then the Shapley effect
for variable $j$ is 
\begin{align*}\phi_j =
\frac1d\sum_{u\subseteq-j}
{d-1\choose |u|}^{-1}
\dfrac{ \cov\bigl(x_j,\bsx_{-u}^\tran\beta_{-u}\mid\bsx_u\bigr)^2}
{\var(x_j\mid\bsx_u)}.
\end{align*}
\end{theorem}
\begin{proof}
See Section~\ref{sec:proofgauslin}.
\end{proof}

A variable with $\beta_j=0$ can still have $\phi_j>0$.
For instance if $\Sigma=\bigl(\begin{smallmatrix}0&\rho\\\rho&0\end{smallmatrix}\bigr)$ and
$f(\bsx)=x_1$, then we can find 
directly from~\eqref{eq:shapfromult}
that $\phi_2=\rho^2/2$ and $\phi_1=1-\rho^2/2$.
For $\rho=\pm1$ we already know this by bijection.

The Shapley value works with conditional variances and the Gaussian
distribution makes these very convenient. For non-Gaussian distributions
the conditional covariance of $\bsx_v$ and $\bsx_w$  given $\bsx_u$ may depend on
the specific value of $\bsx_u$, while in the Gaussian case it is simply
$\Sigma_{vw}-\Sigma_{vu}\Sigma_{uu}^{-1}\Sigma_{uw}$ for all $\bsx_u$.

In a related problem, if we define $\val(u)$ to
be $\var(\sum_{j\in u}x_j)$, instead of $\var( \e(\sum_jx_j\mid \bsx_u))$,
then the Shapley value of variable $j$ is
$\phi_j = \cov(x_j,S)$, where $S=\sum_{j\in1{:}d}x_j$.
See \cite{coli:scar:vacc:2016}.
This quantity can be negative.  For instance, if $d=2$, then
$\phi_1=\var(x_1)+\cov(x_1,x_2)$ which is negative when $x_1$
and $x_2$ are negatively correlated and $x_2$ has much greater
variance than $x_1$.

\subsection{Transformations, bijections and  invariance}\label{sec:transbij}
We can generalize the linear
example to independent random variables
that contribute additively: $f(\bsx) = \sum_{j=1}^dg_j(x_j)$.
Then $\phi_j = \var( g_j(x_j))$.  Replacing $x_j$ by a bijection $\tau_j(x_j)$
and adjusting $g_j$ to $g_j\circ \tau_j^{-1}$ leaves $\phi_j$ unchanged.

More generally, suppose that $y = f(\bsx)$ and we transform the 
variables $x_j$ into $z_j$ by bijections:
$z_j = \tau_j(x_j)$, $x_j = \tau_j^{-1}(z_j)$, for $j=1,\dots,d$. 
Now define $f'(\bsz)= f( \tau_1^{-1}(z_1),\dots,\tau_d^{-1}(z_d))$
and let $\phi_j'$ be the Shapley importance of $z_j$ as a predictor 
of $y'=f'(\bsz)$. Because 
$\var( \e( f'(\bsz)\mid\bsz_u)) = \var( \e( f(\bsx)\mid\bsx_u))$,
we find that $\phi_j'=\phi_j$ for $j=1,\dots,d$,
where $\phi_j$ is the Shapley importance of $x_j$ as a predictor of $y$. 
As a result we can apply invertible transformations to any or all of the $x_j$
without changing the Shapley values.


Now lets revisit the linear setting with an extreme example:
$f(x_1,x_2) = 10^6x_1 +x_2$
with $x_1=10^6x_2$ where $x_2$
(and hence $x_1$) has a finite positive variance. 
Because $\partial f/\partial x_1 \gg \partial f/\partial x_2>0$
and $\var(x_1)\gg\var(x_2)$ one might 
expect  $x_1$ to be the more important variable. 
However,  the Shapley formula easily yields $\phi_1=\phi_2$;
these variables are equally important. This is quite reasonable
because $f$ is a function of $x_1$ alone and equally a function
of $x_2$ alone.

More generally, for $d\ge2$, if 
there is a bijection between any two of the $x_j$ then
those two variables have the same Shapley value.
To see this, let
$x_1 = g_1(x_2)$ and $x_2=g_2(x_1)$, both with probability one 
then for any $u\subset 1{:}d$ with $u\cap\{1,2\}=\emptyset$
we have 
$$\e( f(\bsx)\mid \bsx_{u+\{1\}})=\e( f(\bsx)\mid \bsx_{u+\{2\}}).$$
It follows that
$\ult^2_{u+\{1\}} - \ult^2_u=\ult^2_{u+\{2\}} - \ult^2_u$
and therefore $\phi_1=\phi_2$ by the symmetry property
of Shapley value.

To summarize:
\begin{compactenum}[\quad1)]
\item Shapley value is preserved under invertible transformations, and
\item a bijection between variables implies that they have the same Shapley value.
\end{compactenum}


\subsection{Bivariate settings}\label{sec:any2}

When $d=2$ we can get some simpler formulas for the 
importance of the two variables.

\begin{proposition}\label{prop:anyd2}
Let $f(\bsx)$ have finite variance $\sigma^2>0$ for random $\bsx=(x_1,x_2)$. 
Then from~\eqref{eq:shapfromult},
\begin{align}
\frac{\phi_1}{\sigma^2}
&= \frac12\Bigl( 1 + \frac{\var(\e(Y\mid x_1)) - \var(\e(Y\mid x_2)}{\sigma^2}\Bigr) 
\label{eq:dis2VCE}
\\
&=  \frac12\Bigl( 1 + \frac{\e( \var(Y\mid x_2))-\e( \var(Y\mid x_1))}{\sigma^2}\Bigr),\quad\text{and}
\label{eq:dis2ECV}\\
\frac{\phi_1}{\phi_2} & =
\frac{ \var(\e(Y\mid x_1)) +  \e(\var(Y\mid x_2))}{ \var(\e(Y\mid x_2)) +  \e(\var(Y\mid x_1))}.
\label{eq:dis2sym}
\end{align}
\end{proposition}
\begin{proof}
Using $\ult^2_{\{1,2\}} = \sigma^2$ and $\ult^2_\emptyset=0$,
we find that 
$$
\phi_1 = \frac12\bigl( \ult^2_{\{1\}} + \sigma^2-\ult^2_{\{2\}}\bigr)  
= \frac12\bigl( \sigma^2 + \var(\e(Y\mid x_1)) - \var(\e(Y\mid x_2)),
$$ 
which gives us~\eqref{eq:dis2VCE}. The others are algebraic rearrangements.
\end{proof}

We can use Proposition~\ref{prop:anyd2} to get analogous expressions for
$\phi_2/\sigma^2$ and $\phi_2/\phi_1$ by exchanging indices.

\subsubsection{Farlie-Gumbel-Morgenstern copula for $d=2$}
Here we focus on the case where the dependence between both components $x_1$ and $x_2$ is explicitly described by some copula. There exist simple conditional expectation formulas when considering some classical classes of copulas (see e.g., \cite{crane2008conditional} and references therein). Starting from such formulas, it is possible to derive explicit computations for Shapley values in a linear model. 
In this section, we state explicit results for the Farlie-Gumbel-Morgenstern family of copulas. 

The Farlie-Gumbel-Morgenstern copula describes a random vector $\bsx\in[0,1]^2$
with each component $x_j\sim \dustd[0,1]$ and joint probability density function 
\begin{align}\label{eq:fgmdensity}
c_\theta(x_1,x_2) = 1 + \theta(1-2x_1)(1-2x_2),\quad -1\le\theta\le 1. 
\end{align}
One can show that $\cor(x_1,x_2)=\theta/3$. 
\cite{lai78} proved that, for $0 \leq \theta \leq 1$, $x_1$ and $x_2$ are positively quadrant 
dependent and positively regression dependent. 
Moreover, 
\begin{align}\label{eq:copulalinreg}
\e(x_2\mid x_1)= \frac\theta3x_1+\Bigl(\frac12-\frac\theta6\Bigr).\end{align}
The linearity above is very useful for our purpose, as it will allow an explicit computation for Shapley values in that model. 

\begin{proposition}\label{prop:fgmlin}
Let $f(\bsx) = \bsx^\tran\beta$ for $\bsx,\beta\in\real^2$
and $\bsx\sim c_{\theta}(x_1,x_2)$, with $-1 \leq \theta \leq 1$. 
Then 
$$\frac{\phi_1}{\sigma^2} = \frac12 \left(1+\Bigl(1-\frac{\theta^2}{9}\Bigr) 
\frac{\beta_1^2-\beta_2^2}{12 \sigma^2}\right), 
$$
with $\sigma^2= ({\beta_1^2+ \beta_2^2})/{12}+ {\beta_1\beta_2 \theta}/{18}$. 
\end{proposition}
\begin{proof}
From the linearity of the regression function~\eqref{eq:copulalinreg},
$$\e(f({\bf x})\mid x_1)=x_1 \Bigl(\beta_1+\frac{\theta}{3}\beta_2\Bigr)+
\beta_2 \Bigl(\frac12 - \frac{\theta}{6}\Bigr),$$
thus 
$$\var( \e(f(\bsx)\mid x_1))=\frac{1}{12} \Bigl(\beta_1+ \frac{\theta}{3} \beta_2\Bigr)^2.$$
Symmetry gets us the corresponding expression for 
$\var(\e(f(\bsx)\mid x_2))$. 
Then Proposition~\ref{prop:anyd2} establishes the expression for $\phi_1/\sigma^2$. 
Finally, because $\var(x_j)=1/12$ and $\cor(x_1,x_2)=\theta/3$, we get 
$\sigma^2=(\beta_1^2 + \beta_2^2)/12+ \beta_1 \beta_2 \theta/18$. 
\end{proof}


Now we consider the Farlie-Gumbel-Morgenstern copula, but we assume $x_j$ has as cumulative distribution function $F_j$, and probability density function $F_j'$, not necessarily from the uniform distribution. 

\begin{lemma}\label{lem:crane}
Let $\bsx\in\real^2$ have probability density $F'_1(x_1)F'_2(x_2)c_{\theta}(F_1(x_1),F_2(x_2))$, with $-1 \leq \theta \leq 1$. Then 
$$\e(x_2\mid x_1)=\e(x_2)+\theta (1-2F_1(x_1))\int_{\real} y(1-2F_2(y))F'_2(y)\rd y.$$
For exponential $x_j$ with $F_j(x_j)=1-\exp(-\lambda_jx_j)$ 
for $\lambda_j>0$, we get 
\begin{equation}\label{expmargfgm}
\e(x_2\mid x_1)= \frac{1}{\lambda_2}+ \frac{\theta}{2 \lambda_2}(1-2e^{-\lambda_1x_1}). 
\end{equation}
\end{lemma}
\begin{proof}
\cite{crane2008conditional}. 
\end{proof}

Next we assume that $\bsx$ has exponential margins and we transform these margins to be unit exponential 
by making a corresponding scale adjustment to $\beta$. From 
Section~\ref{sec:transbij},
we know that such transformations do 
not change the Shapley value. 

\begin{proposition}\label{propexpfgm}
Let $f(\bsx) = \bsx^\tran\beta$ for $\bsx,\beta\in\real^2$
where $\bsx$ has probability density function 
$e^{-x_1-x_2} c_{\theta}(1- e^{-x_1},1- e^{-x_2})$, where  $-1 \leq \theta \leq 1.$
Then 
\begin{align}\label{eq:expfgm}
\frac{\phi_1}{\sigma^2} = \frac12 \Bigl(1+\Bigl(1-\frac{\theta^2}{12}\Bigr)\frac{\beta_1^2-\beta_2^2}{ \sigma^2}\Bigr) 
\end{align}
with $\sigma^2= \beta_1^2+\beta_2^2+\theta\beta_1\beta_2/2$. 
\end{proposition}
\begin{proof}
From Lemma~\ref{lem:crane}, 
$\e(x_2\mid x_1)=1+\theta/2-\theta e^{-x_1}$
so 
$$\e(f(\bsx)\mid x_1) = \beta_1x_1+\beta_2(1+\theta/2-\theta e^{-x_1}).$$
Therefore 
$$\var(\e(f(\bsx)\mid x_1)) = \beta_1^2 
+\beta_2^2\theta^2\var(e^{-x_1}) 
-2\beta_1\beta_2\theta\cov(x_1,e^{-x_1}). 
$$
Now $\var(e^{-x_1})=\e(e^{-2x_1})-\e(e^{-x_1})^2=1/12$
and 
$$
\cov(x_1,e^{-x_1})=\int_0^\infty xe^{-2x}\rd x -\frac12=-\frac14,
$$
so $\var(\e(f(\bsx)\mid x_1)) = \beta_1^2+\beta_2^2\theta^2/12+\beta_1\beta_2\theta/2$. 
This establishes~\eqref{eq:expfgm} by Proposition~\ref{prop:anyd2}. 
\end{proof}

Suppose that $\beta_1>\beta_2>0$. Then of course $\phi_1/\sigma^2>1/2$. 
Equation~\eqref{eq:expfgm} shows that $\phi_1/\sigma^2$ decreases 
as $\theta$ increases from $0$ to $1$. It does not approach $1/2$ because 
even at $\theta=1$, $x_2$ is not a deterministic function of $x_1$.

\subsubsection{Gaussian variables, exponential $f$, $d=2$}\label{gausframenl}

Let $\bsx\sim\dnorm(\mu,\Sigma)$ and take $Y=e^{\beta_0+\sum_{j=1}^dx_j\beta_j}$.
The effect of $\beta_0$ and $\mu_j$ is simply to scale $Y$ and so
we can take $\beta_0=0$ and $\mu=0$ without affecting $\phi_j/\sigma^2$.
Next we suppose that the diagonal elements of $\Sigma$ are nonzero.
By the transformation result in
Section~\ref{sec:transbij} 
we can replace
each $x_j$ by $x_j/\Sigma_{jj}$ if need be without changing $\phi_j$
and so we suppose that each $x_j\sim\dnorm(0,1)$. Here we find
variable importances for $d=2$.


\begin{proposition}\label{prop:gauslinexp}
Let $f(\bsx) = \exp\bigl(\bsx^\tran\beta\bigr)$ for $\bsx,\beta\in\real^2$
and $\bsx\sim\dnorm({\bf 0},\Sigma)$, for $\Sigma = \begin{pmatrix}
   1 & \rho \\
   \rho & 1 
\end{pmatrix}$. Then 
\begin{align}\label{eq:gauslinexp}
\frac{\phi_1}{\sigma^2} = 
\frac12 \biggl(1+
\frac{e^{(\beta_1+ \beta_2\rho)^2}-e^{(\beta_2+ \beta_1\rho)^2}}
{e^{\beta_1^2+\beta_2^2+2\rho\beta_1\beta_2}-1}\biggr),
\end{align}
where the variance of $f(\bsx)$ is
\begin{align}\label{eq:varexpo}
\sigma^2=e^{\beta_1^2+\beta_2^2+2 \rho \beta_1\beta_2}(e^{\beta_1^2+\beta_2^2+2 \rho \beta_1\beta_2} -1).
\end{align}
\end{proposition}
\begin{proof}
Recall the lognormal moments: if $Z\sim\dnorm(\mu,\sigma^2)$
then $\e(e^Z)= e^{\mu+\sigma^2/2}$ and
$\var(e^Z) = (e^{\sigma^2}-1)e^{2\mu+\sigma^2}$.
Taking $Z=\bsx^\tran\beta$ we find that $Y=e^Z$ has
variance $\sigma^2$ given by~\eqref{eq:varexpo}.

The distribution of $x_2\beta_2$ given $x_1$ is 
$\dnorm(\rho x_1\beta_2,(1-\rho^2)\beta^2_2)$.
Therefore 
\begin{align*}
\e(Y\mid x_1) &= e^{(\beta_1+\rho\beta_2)x_1+\beta_2^2(1-\rho^2)/2},\quad\text{and so}\\
\var(\e(Y\mid x_1)) &= e^{\beta_2^2(1-\rho^2)}
e^{(\beta_1+\rho\beta_2)^2}
(e^{(\beta_1+\rho\beta_2)^2}-1)\\
& = e^{\beta^\tran\Sigma\beta}(e^{(\beta_1+\rho\beta_2)^2}-1).
\end{align*}
Similarly, $\var(\e(Y\mid x_2))
= e^{\beta^\tran\Sigma\beta}(e^{(\beta_2+\rho\beta_1)^2}-1)$.
Then applying Proposition~\ref{prop:anyd2} and noticing
that the lead factor $e^{\beta^\tran\Sigma\beta}$ appears
also in $\sigma^2$, yields the result.
\end{proof}

If $\rho=\pm1$ then $\phi_1/\sigma^2=1/2$ as it must because
there is then a bijection between the variables.
The value of $\phi_1/\sigma^2$ in~\eqref{eq:gauslinexp} is
unchanged if we replace $\rho$ by $-\rho$.  The formula is
not obviously symmetric, but the fraction within parentheses
there can be divided by the corresponding one for $-\rho$
and the ratio reduces to $1$. More directly, we know from
Section~\ref{sec:transbij} that making
the transformation $x_2\to-x_2$ and $\beta_2\to-\beta_2$ would
leave the variable importances unchanged while switching $\rho\to-\rho$.

It is clear that for $\beta_1>\beta_2$ we must have $\phi_1/\sigma^2\ge1/2$.
Even with the closed form~\eqref{eq:gauslinexp}, it is not obvious
how $\phi_1/\sigma^2$ should depend on $\rho$ or on $\beta$. 
Figure~\ref{fig:lognormalphi}
shows that increasing $|\rho|$ from zero generally raises the importance of
$x_1$ until at some high correlation level the relative importance quickly drops down to $1/2$.  Also, for $\rho=0$ the effect 
of $\beta_1$ over the range $2\le\beta_1\le 8$ is quite small when $\beta_2=1$.

\begin{figure}
\centering
\includegraphics[width=\hsize]{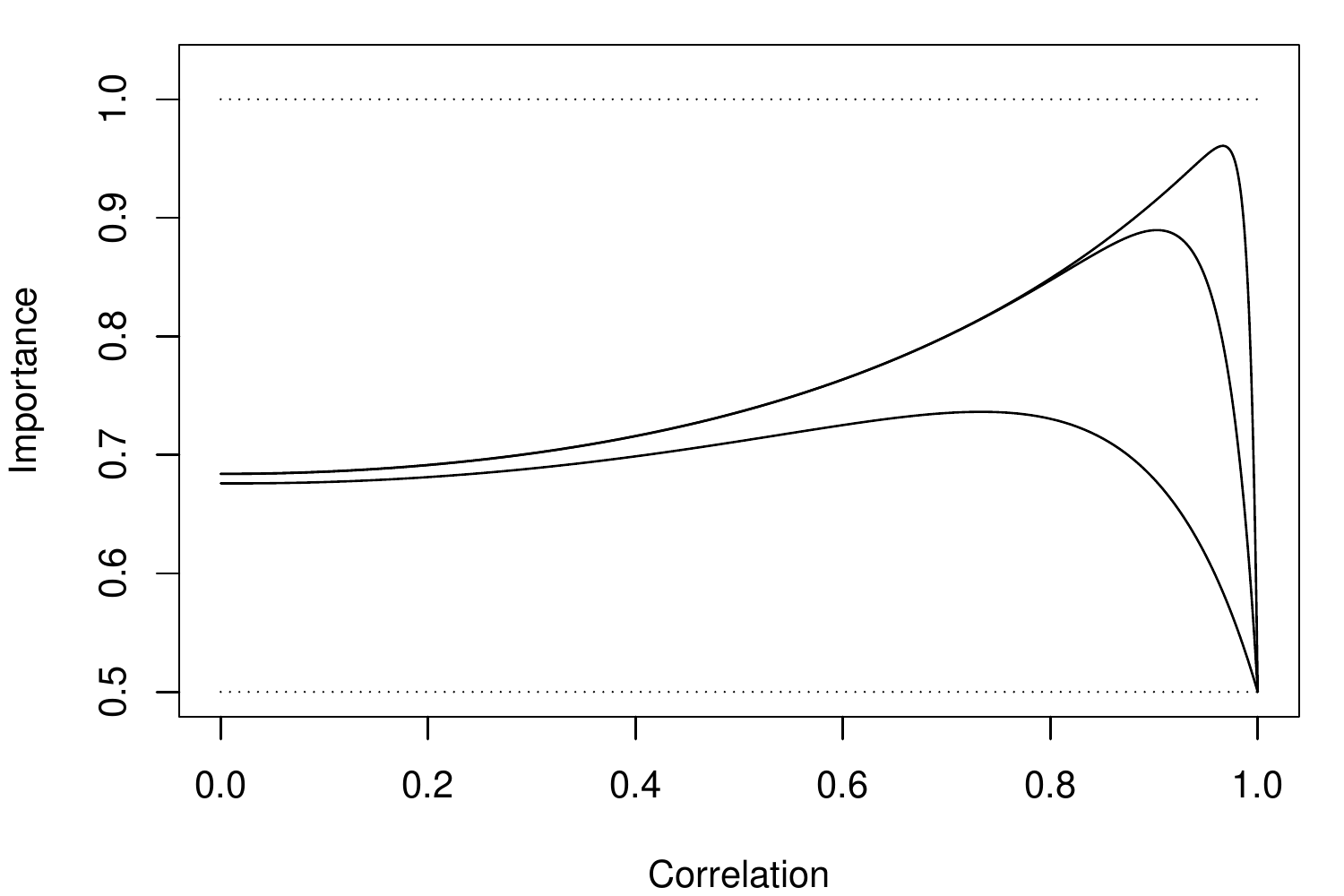}
\caption{\label{fig:lognormalphi}
Relative importance $\phi_1/\sigma^2$ versus correlation $|\rho|$
from Proposition~\ref{prop:fgmlin}. From top to bottom,
$\beta^\tran$ is $(8,1)$, $(4,1)$, and $(2,1)$.
}
\end{figure}

The lognormal case is different from the bivariate normal case. There, the value of $\phi_1$
converges monotonically towards $1/2$ as $|\rho|$ increases from $0$ to $1$.
\subsection{Holes}

Here we consider the simplest setting where there is an unreachable
part of the $\bsx$ space.
We consider two binary variables $x_1$ and $x_2$ but
$x_1=x_2=1$ never occurs.  For instance $f$ could be the
weight of a sea turtle, $x_1$ could be $1$ iff the turtle is
bearing eggs and $x_2$ could be $1$ iff the turtle is male.
It may seem unreasonable to even attempt to
compare the importance of these variables (male/female
versus eggs/none) but Shapley
value does provide such a comparison based on compelling axioms
in the event that we do seek a comparison.

\begin{table}[t]
\centering
\begin{tabular}{llll}
\toprule
$p$ & $x_1$ & $x_2$ & $y$\\
\midrule
$p_0$& $0$ & $0$ & $y_0$\\
$p_1$ & $1$ & $0$ &$y_1$\\
$p_2$& $0$ & $1$ &$y_2$\\
\bottomrule
\end{tabular}
\caption{\label{tab:3pt}
The random variable $y=f(\bsx)$ is the given function
of $\bsx=(x_1,x_2)$. That vector takes three values
with the probabilities in this table. For example,
$\Pr( \bsx=(1,0))=p_1$ and then $y=y_1$.
}
\end{table}
This simplest setting 
is depicted in Table~\ref{tab:3pt}
where $p_0+p_1+p_2=1$.
We assume that $p_1>0$ and $p_2>0$ for otherwise
the function does not have two input variables.

\begin{theorem}\label{thm:turtle}
Let $y$ be a function of the random vector
$\bsx$ as given in Table~\ref{tab:3pt}.
Assume that $\sigma^2=\var(y)>0$,
and $\min(p_1,p_2)>0$. Then the Shapley relative importance
of variable $x_1$ is
\begin{align}\label{eq:thmturtle}
\frac12\Bigl( 
1 + \frac{p_0}{\sigma^2} \times  
\frac{ 
p_1(1-p_1)\bar y_1^2-p_2(1-p_2) \bar y_2^2  
}
{(1-p_1)(1-p_2)}
\Bigr)  
\end{align} 
where $\bar y_j=y_j-y_0$ for $j=1,2$.
\end{theorem}
\begin{proof}
See section~\ref{sec:turtleproof}.
\end{proof}

We see that when $p_0=0$, then the Shapley relative importance
of $x_1$ is $1/2$.  That is what it must be because there is then
a bijection between $x_1$ and $x_2$ via $x_1+x_2=1$.

Now suppose that  $\bar y_1=\bar y_2$.  For instance
$y_1=y_2=1$ while $y_0=0$. 
Then the more important variable is
the one with the larger variance.  That is $x_1$ is more
important if $p_1(1-p_1)> p_2(1-p_2)$.  This can only
happen if $p_1>p_2$.
So the more probable input is the more important one in this case.



\subsection{Maximum of exponential random variables}

 \cite{kein:hilg:jeil:rupp:2004} considered a network of neurons
 $e_1,\dots,e_d$ where the $e_j$ have independent lifetimes $x_j$ that are 
 exponentially distributed with mean $1/\lambda_j$.
In their setting the value of a set of neurons is $\phi(u) = \e( \max_{j\in u}x_j)$,
 that is the expected amount of time that at least part of that subset survives.
 For $d=3$, they give a Shapley value of 
 $$
 \phi_j=\frac1{\lambda_1}
 -\frac12\frac1{\lambda_1+\lambda_2}
 -\frac12\frac1{\lambda_1+\lambda_3}
 +\frac13\frac1{\lambda_1+\lambda_2+\lambda_3},
 $$
but they do not give a proof. While value in this example is not based on
prediction error, we include it because it is another example of a closed
form for Shapley value based on random variables.
We prove their formula here 
and generalize it to any $d\ge1$.

\begin{theorem}\label{thm:exprv}
Let the value of a set $u\subseteq1{:}d$ be
$\val(u) = \e(\max_{j\in u}x_j)$ where $x_1,\dots,x_d$
are independent exponential random variables with
$\e(x_j) = 1/\lambda_j$.
Then
$$
\phi_j =
\sum_{r=1}^d\frac{(-1)^{r-1}}r
\sum_{w\subseteq1{:}d,j\in w,|w|=r}
\frac1{\sum_{\ell\in w}\lambda_\ell}.
$$
\end{theorem}
\begin{proof}
See section~\ref{sec:proveexprv}.
\end{proof}

\section{Conclusions}\label{sec:conc}

The Shapley value from economics remedies the conceptual
difficulties in measuring importance of dependent variables
via ANOVA. Like ANOVA it uses variances, but unlike the dependent
data ANOVA, Shapley value never goes negative and it can be
defined without onerous assumptions on the input distribution.

We find that Shapley value has useful properties. When two variables
are functionally equivalent, then they get equal Shapley value. When
an invertible transformation is made to a variable, it retains its Shapley value.
We thus conclude that \cite{song:nels:staum:2016} had the right
idea proposing Shapley value for dependent inputs.
Computation of Shapley values remains a challenge outside of special cases like 
the ones we discuss here.

A potential application that we find interesting is measuring the
importance of parameters in a Bayesian context.  When the parameter
vector $\beta$ has an approximate Gaussian posterior distribution, as the central limit theorem
often provides, then Theorem~\ref{thm:gauslin} yields a measure $\phi_j(\bsx_0)$
for the importance of parameter $\beta_j$ for the posterior uncertainty of the
prediction $\bsx_0^\tran\beta$.  We hasten to add that parameter independence is quite different from variable importance, which is a more common goal.  By this measure an
important parameter is one whose uncertainty dominates uncertainty in $\bsx_0^\tran\beta$. The corresponding variable may or may not be important.
Another potential application is in modeling the importance of order
statistics. They naturally belong to a non-rectangular set~\cite{lebr:dutf:2014}.

\section*{Acknowledgments}

This work was supported by grant  DMS-1521145
from the U.S.\ National Science Foundation.
We thank Marco Scarsini, Jiangming Xiang, Bertrand Iooss,
two anonymous referees and an associate editor for valuable comments.

\bibliographystyle{apalike}
\bibliography{sensitivity}

\section{Proofs}\label{sec:proofs}

\subsection{Proof of Theorem~\ref{thm:gauslin}}\label{sec:proofgauslin}
Recall that $f(\bsx) = \bsx^\tran\beta$ where $\bsx\sim\dnorm(\mu,\Sigma)$.
We also assumed that $\Sigma$ is of full rank.
Now
$
\var(\bsx_{-u}\mid\bsx_u)
= \Sigma_{-u,-u}-\Sigma_{-u,u}\Sigma_{u,u}^{-1}\Sigma_{u,-u},
$
and so
\begin{align*}
\var( f(\bsx)\mid\bsx_u)
&=\var(\bsx_u^\tran\beta_u+\bsx_{-u}^\tran\beta_{-u}\mid\bsx_u)\\
&=\var(\bsx_{-u}^\tran\beta_{-u}\mid\bsx_u)\\
&=\beta_{-u}^\tran 
\bigl(\Sigma_{-u,-u}-\Sigma_{-u,u}\Sigma_{u,u}^{-1}\Sigma_{u,-u}\bigr) 
\beta_{-u}.
\end{align*}

We will use $v=v(j,u) \equiv -u-\{j\}$.
It helps to visualize the partitioned covariance matrix
$$
\Sigma = \begin{pmatrix}
\Sigma_{uu} & \Sigma_{uj} & \Sigma_{uv}\\
\Sigma_{ju} & \Sigma_{jj} & \Sigma_{jv}\\
\Sigma_{vu} & \Sigma_{vj} & \Sigma_{vv}\\
\end{pmatrix}
$$
if the indices have been ordered for those in $u$ to precede $j$ which precedes
those in $v$.
For this section only, we make a further notational compression
shortening $u+\{j\}$ to $u+j$.
Next
\begin{align*}
\ult^2_{u+j}-\ult_u^2
&=\var( f(\bsx)\mid \bsx_u)-\var(f(\bsx)\mid \bsx_{u+j})\\
&=
\beta_{-u}^\tran\bigl(\Sigma_{-u,-u}-\Sigma_{-u,u}\Sigma_{u,u}^{-1}\Sigma_{u,-u}\bigr)\beta_{-u}\\
&\phe\,-\beta_v^\tran\bigl(\Sigma_{vv}-\Sigma_{v,u+j}\Sigma_{u+j,u+j}^{-1}\Sigma_{u+j,v}\bigr)\beta_v.
\end{align*}

Using the formula for the inverse of a partitioned matrix, we find that
$$
\Sigma_{u+j,u+j}^{-1} =
\begin{pmatrix}
\Sigma_{uu}^{-1}+\Sigma_{uu}^{-1}\Sigma_{uj}D_j(u)\Sigma_{ju}\Sigma_{uu}^{-1}
& -\Sigma_{uu}^{-1}\Sigma_{uj}D_j(u)\\
-D_j(u)\Sigma_{ju}\Sigma_{uu}^{-1} & D_j(u) 
\end{pmatrix},
$$
where $D_j(u) = (\Sigma_{jj}-\Sigma_{ju}\Sigma_{uu}^{-1}\Sigma_{uj})^{-1}
=\var(x_j\mid \bsx_u)^{-1}$, which exists because $\Sigma$ has full rank.
Continuing,
\begin{align*}
&\Sigma_{v,u+j}\Sigma_{u+j,u+j}^{-1}\Sigma_{u+j,v}\\
& = 
\begin{pmatrix}
\Sigma_{vu} & \Sigma_{vj}
\end{pmatrix}
\begin{pmatrix}
\Sigma_{uu}^{-1}+\Sigma_{uu}^{-1}\Sigma_{uj}D_j(u)\Sigma_{ju}\Sigma_{uu}^{-1}
& -\Sigma_{uu}^{-1}\Sigma_{uj}D_j(u)\\
-D_j(u)\Sigma_{ju}\Sigma_{uu}^{-1} & D_j(u) 
\end{pmatrix}
\begin{pmatrix}
\Sigma_{uv}\\
\Sigma_{jv}
\end{pmatrix}\\
&=
\begin{pmatrix}
\Sigma_{vu} & \Sigma_{vj}
\end{pmatrix}
\begin{pmatrix}
\Sigma_{uu}^{-1}\Sigma_{uv}+\Sigma_{uu}^{-1}\Sigma_{uj}D_j(u)\Sigma_{ju}\Sigma_{uu}^{-1}\Sigma_{uv} -\Sigma_{uu}^{-1}\Sigma_{uj}D_j(u)\Sigma_{jv}\\
-D_j(u)\Sigma_{ju}\Sigma_{uu}^{-1}\Sigma_{uv} + D_j(u) \Sigma_{jv}
\end{pmatrix}\\
&=
\Sigma_{vu}\Sigma_{uu}^{-1}\Sigma_{uv}+\Sigma_{vu}\Sigma_{uu}^{-1}\Sigma_{uj}D_j(u)\Sigma_{ju}\Sigma_{uu}^{-1}\Sigma_{uv} -\Sigma_{vu}\Sigma_{uu}^{-1}\Sigma_{uj}D_j(u)\Sigma_{jv}\\
&\phe\,-
\Sigma_{vj}D_j(u)\Sigma_{ju}\Sigma_{uu}^{-1}\Sigma_{uv}+\Sigma_{vj}D_j(u)\Sigma_{jv}\\
&=\Sigma_{vu}\Sigma_{uu}^{-1}\Sigma_{uv}
+D_j(u)\,\bigl(\Sigma_{vu}\Sigma_{uu}^{-1}\Sigma_{uj}-\Sigma_{vj}\bigr)
\bigl(\Sigma_{ju}\Sigma_{uu}^{-1}\Sigma_{uv}-\Sigma_{jv}\bigr)\\
&=
\Sigma_{vu}\Sigma_{uu}^{-1}\Sigma_{uv}
+D_j(u) \cov( \bsx_v,x_j\mid\bsx_u) \cov( x_j,\bsx_v\mid\bsx_u)
\end{align*}
recalling that $D_j(u)$ is a scalar.

Now $\ult_{u+j}^2-\ult_u^2$ is 
\begin{align*}
&\beta_{-u}^\tran\cov(\bsx_{-u}\mid\bsx_u)\beta_{-u}-\beta_v^\tran\Sigma_{vv}\beta_v\\
&\phe+\beta_v^\tran\bigl(
\Sigma_{vu}\Sigma_{uu}^{-1}\Sigma_{uv}
+D_j(u) \cov( \bsx_v,x_j\mid\bsx_u) \cov( x_j,\bsx_v\mid\bsx_u)
\bigr)\beta_v\\
&=
\beta_{-u}^\tran\cov(\bsx_{-u}\mid\bsx_u)\beta_{-u}-\beta_v^\tran\cov(\bsx_v\mid\bsx_u)
\beta_v\\
&\phe+D_j(u) \beta_v^\tran\cov( \bsx_v,x_j\mid\bsx_u) \cov( x_j,\bsx_v\mid\bsx_u)\beta_v\\
&=\Sigma_{jj}\beta_j^2+\beta_j\Sigma_{jv}\beta_v+\beta_v^\tran\Sigma_{vj}\beta_j\\
&\phe -\beta_j^2\Sigma_{ju}\Sigma_{uu}^{-1}\Sigma_{uj}
-\beta_j\Sigma_{ju}\Sigma_{uu}^{-1}\Sigma_{uv}\beta_v
-\beta_v^\tran\Sigma_{vu}\Sigma_{uu}^{-1}\Sigma_{uj}\beta_j\\
&\phe+D_j(u) \beta_v^\tran\cov( \bsx_v,x_j\mid\bsx_u) \cov( x_j,\bsx_v\mid\bsx_u)\beta_v\\
&=\beta_j^2\var(x_j\mid\bsx_u) + 2\beta_j\cov(x_j,\bsx_v\mid\bsx_u)\beta_v\\
&\phe
+D_j(u) \beta_v^\tran\cov( \bsx_v,x_j\mid\bsx_u) \cov( x_j,\bsx_v\mid\bsx_u)\beta_v.
\end{align*}
Putting this together, the Shapley value of variable $j$ is
\begin{equation}\label{eq:oldexpr}
\begin{split}
\phi_j 
&=
\frac1d\sum_{u\subseteq-j}
{d-1\choose |u|}^{-1}\Bigl(
\beta_j^2\var(x_j\mid\bsx_u) + 2\beta_j\cov(x_j,\bsx_{-u-j}\mid\bsx_u)\beta_{-u-j}\\
&\phe +
\var(x_j\mid\bsx_u)^{-1}
\beta_{-u-j}^\tran\cov( \bsx_{-u-j},x_j\mid\bsx_u) \cov( x_j,\bsx_{-u-j}\mid\bsx_u)\beta_{-u-j}\Bigr). 
\end{split}
\end{equation}

Writing
\begin{align*}
\cov(x_j,\bsx_{-u}^\tran\beta_{-u}\mid\bsx_u)
&=\cov(x_j,\bsx_{-u-j}^\tran\beta_{-u-j}\mid\bsx_u)+\beta_j\var(x_j\mid\bsx_u)
\end{align*}
we then find that
$\cov(x_j,\bsx_{-u}^\tran\beta_{-u}\mid\bsx_u)^2/\var(x_j\mid\bsx_u)$
equals the factor to the right of ${d-1\choose |u|}$ in~\eqref{eq:oldexpr}.

\subsection{Proof of Theorem~\ref{thm:turtle}}\label{sec:turtleproof}

Without loss of generality take $y_0=0$.
Then $\mu = p_1y_1 +p_2y_2$
and $\sigma^2  = p_1y_1^2+p_2y_2^2-\mu^2$.

Now with $y_0=0$,
\begin{align*}
\var(\e(  y\mid x_1)) 
&= (p_0+p_2)\Bigl(
\dfrac{y_2p_2}{p_0+p_2}-\mu\Bigr)^2 + 
p_1(y_1-\mu)^2\\
&= (1-p_1)\Bigl(
\dfrac{y_2p_2}{1-p_1}-\mu\Bigr)^2 + 
p_1(y_1-\mu)^2\\
&= \frac{p_2^2y_2^2}{1-p_1}
-2\mu y_2p_2+\mu^2(1-p_1)
+p_1(y_1-\mu)^2\\
&= \frac{p_2^2y_2^2}{1-p_1}
-2
(p_1y_1 +p_2y_2)y_2p_2+
(p_1y_1 +p_2y_2)^2(1-p_1)
+p_1(y_1(1-p_1)-p_2y_2)^2\\
&=
y_2^2\Bigl( 
\frac{p_2^2}{1-p_1} - 2p_2^2 + p_2^2(1-p_1) + p_1p_2^2 
\Bigr)\\
&\quad+
y_1^2\Bigl( p_1^2(1-p_1) + p_1(1-p_1)^2 
\Bigr)\\
&\quad+
y_1y_2\Bigl( 
-2p_1p_2 + 2p_1p_2(1-p_1) - 2p_1p_2(1-p_1) 
 \Bigr)\\
&=y_2^2\Bigl( \frac{p_2^2}{1-p_1} - p_2^2\Bigr)
+y_1^2p_1(1-p_1)-2y_1y_2 p_1p_2\\
&=y_2^2\frac{p_1p_2^2}{1-p_1}+y_1^2p_1(1-p_1)-2y_1y_2 p_1p_2.
\end{align*}

Then
$\var(\e(  y\mid x_1)) -\var(\e(  y\mid x_2)) $
equals
\begin{align*}
&\quad y_2^2\frac{p_1p_2^2}{1-p_1}
+y_1^2p_1(1-p_1)
-y_1^2\frac{p_2p_1^2}{1-p_2}
-y_2^2p_2(1-p_2)\\
&=
y_2^2\Bigl(
\frac{p_1p_2^2}{1-p_1} - p_2(1-p_2) 
\Bigr) 
+y_1^2\Bigl(
p_1(1-p_1)-\frac{p_2p_1^2}{1-p_2}
\Bigr)\\
&=
y_1^2\Bigl(\frac{p_0p_1}{1-p_2}\Bigr) -y_2^2\Bigl(\frac{p_0p_2}{1-p_1}\Bigr) .
\end{align*}

Finally, the relative importance of variable $x_1$ is 
\begin{align*}
\frac12\Bigl(1+
\frac{y_1^2\bigl(\frac{p_0p_1}{1-p_2}\bigr) -y_2^2\bigl(\frac{p_0p_2}{1-p_1}\bigr) }
{\sigma^2}\Bigr)
&=
\frac12\Bigl(1+\frac{p_0}{\sigma^2}
\frac{y_1^2p_1(1-p_1) -y_2^2p_2(1-p_2)}
{(1-p_1)(1-p_2)}\Bigr)\\
&=
\frac12\Bigl(1+\frac{p_0}{\sigma^2}
\Bigl(\frac{p_1y_1^2}{1-p_2}-\frac{p_2y_2^2}{1-p_1}\Bigr)\Bigr).
\end{align*}

\subsection{Proof of Theorem~\ref{thm:exprv}}\label{sec:proveexprv}
Recall that the random vector $\bsx\in[0,\infty)^d$ has independent 
components $x_j$. They are exponentially 
distributed and $\e(x_j)=1/\lambda_j$ for $0<\lambda_j<\infty$. 
Let $M_u = \max_{j\in u}x_j$ and define value 
$\val(u) = \e(M_u)$. 
Our first step is to evaluate the expected value of
a maximum of independent not identically distributed exponential
random variables.

\begin{proposition}\label{prop:emu}
$$\e(M_u) =
\sum_{\emptyset\ne v\subseteq u}(-1)^{|v|-1}\frac1{\sum_{j\in v}\lambda_j}.$$
\end{proposition}
\begin{proof}
First $\Pr( M_u  < x ) = \prod_{j\in u}\Pr(x_j<x) = \prod_{j\in u}(1-e^{-\lambda_jx})$. Then,
\begin{align*}
\e( M_u ) 
&= \int_0^\infty\left(1-\prod_{j\in u}(1-e^{-\lambda_jx})\right) \rd x\\
&= \int_0^\infty \left( 1-\sum_{v\subseteq u}(-e^{-\lambda_jx})\right) \rd x\\
&= \sum_{\emptyset\ne v\subseteq u}(-1)^{|v|-1}\int_0^\infty e^{-x\sum_{j\in v}\lambda_j}\rd x\\
&= \sum_{\emptyset\ne v\subseteq u}(-1)^{|v|-1}\frac1{\sum_{j\in v}\lambda_j}. 
\qedhere 
\end{align*}
\end{proof}

Using Proposition~\ref{prop:emu} we get Shapley value 
\begin{align*}
\phi_j 
& = \frac1d\sum_{u\subseteq -\{j\}}{d-1\choose |u|}^{-1}(\val(u+j)-\val(u))\\
& = \frac1d\sum_{u\subseteq -\{j\}}{d-1\choose |u|}^{-1}
\Biggl(\sum_{\emptyset\ne v\subseteq u+j}(-1)^{|v|-1}\frac1{\sum_{\ell\in v}\lambda_\ell}
-\sum_{\emptyset\ne v\subseteq u}(-1)^{|v|-1}\frac1{\sum_{\ell \in v}\lambda_\ell}\Biggr)\\
& = \frac1d\sum_{u\subseteq -\{j\}}{d-1\choose |u|}^{-1}
\sum_{w\subseteq u}(-1)^{|w|}\frac1{\sum_{\ell\in w+j}\lambda_\ell}. 
\end{align*}
Introducing the `slack variable' $v$ with $u=v+w$,
\begin{align*}
\phi_j 
 & = \frac1d\sum_{w\subseteq -\{j\}}
(-1)^{|w|}\frac1{\sum_{\ell\in w+j}\lambda_\ell}
\sum_{v\subseteq -\{j\}-w}{d-1\choose |v+w|}^{-1}\\
& = \frac1d\sum_{w\subseteq -\{j\}}
(-1)^{|w|}\frac1{\sum_{\ell\in w+j}\lambda_\ell}
\sum_{r=0}^{d-1-|w|}{d-1-|w| \choose r}\Bigm/{d-1\choose r+|w|}\\
& = \frac1d\sum_{w{:}j\in w}
(-1)^{|w|-1}\frac1{\sum_{\ell\in w}\lambda_\ell}
\sum_{r=0}^{d-|w|}{d-|w| \choose r}\Bigm/{d-1\choose r-1+|w|}. 
\end{align*}
The following diagonal sum identity for binomial 
coefficients will be useful:
$$
\sum_{r=0}^A{L+r\choose r} = {A+L+1\choose A}. 
$$
Using that identity at the third step below,
\begin{align*}
\sum_{r=0}^{d-|w|}{d-|w| \choose r}\Bigm/{d-1\choose r-1+|w|}
&=
\sum_{r=0}^{d-|w|}\frac{ (d-|w|)!}{r!}
\Bigm/\frac{ (d-1)!}{(r-1+|w|)!}\\
&=\frac{(d-|w|)!}{(d-1)!}(|w|-1)!\sum_{r=0}^{d-|w|}{r-1+|w|\choose r}\\
&=\frac{(d-|w|)!}{(d-1)!}(|w|-1)!{d \choose d-|w|}\\
&=\frac{d}{|w|}. 
\end{align*}
As a result,
$$\phi_j = \sum_{w{:}j\in w}
\frac1{|w|} (-1)^{|w|-1}
\frac1{\sum_{\ell\in w}\lambda_\ell}$$
which after slight rearrangement gives the conclusion
of Theorem~\ref{thm:exprv}.
\end{document}